\documentclass{amsart}
\usepackage{amsmath}
\usepackage{amssymb}
\usepackage{amsfonts}
\usepackage{float}
\usepackage{fancyhdr}

\newtheorem{theorem}{Theorem}[section]
\theoremstyle{plain}

\newtheorem{notation}{Notation}

\newtheorem{lemma}[theorem]{Lemma}

\newtheorem{proposition}[theorem]{Proposition}

\numberwithin{equation}{section}
\input{tcilatex}
\begin{document}
\title[Products of point stabilizers]{Primitive permutation groups as
products of point stabilizers}
\author{Martino Garonzi}
\address[Martino Garonzi]{Departamento de Matematica, Universidade de Bras%
\'{\i}lia, Campus Universit\'{a}rio Darcy Ribeiro, Bras\'{\i}lia - DF
70910-900, Brasil}
\email{mgaronzi@gmail.com}
\thanks{MG acknowledges the support of the University of Brasilia}
\author{Dan Levy}
\address[Dan Levy]{The School of Computer Sciences, The Academic College of
Tel-Aviv-Yaffo, 2 Rabenu Yeruham St., Tel-Aviv 61083, Israel}
\email{danlevy@mta.ac.il}
\author{Attila Mar\'oti}
\address[Attila Mar\'oti]{MTA Alfr\'ed R\'enyi Institute of Mathematics,
Re\'altanoda utca 13-15, H-1053, Budapest, Hungary}
\email{maroti.attila@renyi.mta.hu}
\thanks{AM was supported by the MTA R\'{e}nyi Lend\"{u}let Groups and Graphs
Research Group and by OTKA K84233.}
\author{Iulian I. Simion}
\address[Iulian I. Simion]{Department of Mathematics, University of Padova,
Via Trieste 63, 35121 Padova, Italy}
\email{iulian.simion@math.unipd.it}
\thanks{IS acknowledges the support of the University of Padova (grants
CPDR131579/13 and CPDA125818/12). }
\date{\today }

\begin{abstract}
We prove that there exists a universal constant $c$ such that any finite
primitive permutation group of degree $n$ with a non-trivial point
stabilizer is a product of no more than $c\log n$ point stabilizers.
\end{abstract}

\maketitle

\section{Introduction}

Given a finite group $G$ \footnote{%
All groups discussed are assumed to be finite.}\ and a subgroup $H$ of $G$
whose normal closure is $G$, one can show, by a straightforward elementary
argument, that $G$ is the setwise product of at least $\frac{\log \left\vert
G\right\vert }{\log \left\vert H\right\vert }$ conjugates of $H$. A far
reaching conjecture of Liebeck, Nikolov and Shalev states \cite{LNS} that in
the case that $G$ is a non-abelian simple group, $\frac{\log \left\vert
G\right\vert }{\log \left\vert H\right\vert }$ is in fact the right order of
magnitude for the minimal number of conjugates of $H$ whose product is $G$,
namely, there exists a universal constant $c$ such that for any non-abelian
simple group $G$ and any non-trivial $H\leq G$, the group $G$ is the product
of no more than $c\frac{\log \left\vert G\right\vert }{\log \left\vert
H\right\vert }$conjugates of $H$. Later on, in \cite{LNS2}, this conjecture
was extended to allow $H$ to be any subset of $G$ of size at least $2$. Some
weaker versions of these conjectures are proved in \cite[Theorem 2]{LNS}, 
\cite[Theorem 3]{LNS2}, and \cite[Theorem 1.3]{GPSSBoundedProof2013}.

Here we look for a universal upper bound on the minimal length of a product
covering of a finite primitive permutation group by conjugates of a point
stabilizer. We will prove the following logarithmic\footnote{%
Throughout the paper, $\log $ stands for logarithm in base $2$.} bound:

\smallskip

{}\textbf{Theorem 1. }\label{Th_GeneralPrimitive} \textit{There exists a \
universal constant }$c$\textit{\ such that if }$G$\textit{\ is any primitive
permutation group of degree }$n$\textit{\ with a non-trivial point
stabilizer }$H$\textit{\ then }$G$\textit{\ is a product of at most }$c\log
n $\textit{\ conjugates of }$H$\textit{.}

\smallskip

Note that in most relevant cases, $\frac{\log \left\vert G\right\vert }{\log
\left\vert H\right\vert }<$ $\log \left\vert G:H\right\vert =$ $\log n$ (see
Lemma \ref{Lem_LNSTranslateToIndex}). Thus we do not know whether the bound
provided by Theorem \ref{Th_GeneralPrimitive} is the best possible. In fact,
on the basis of currently published results we don't even know if this bound
can be improved for any particular O'Nan-Scott family of primitive groups.
We believe that these questions deserve further investigation.

\section{Preliminaries}

We collect some preparatory results and notation.

\begin{lemma}
\label{Lem_LNSTranslateToIndex}Let $G$ be a group and $H\leq G$ such that $%
\left\vert H\right\vert \geq 4$ and $\left\vert G:H\right\vert \geq 4$. Then 
$\log \left\vert G\right\vert /\log \left\vert H\right\vert \leq \log
\left\vert G:H\right\vert $.
\end{lemma}

\begin{proof}
Set $x:=\log \left\vert G\right\vert $ and $y:=\log \left\vert H\right\vert $%
. Then the desired inequality reads $x/y\leq x-y$, which is equivalent to $%
x\geq y+1+\frac{1}{y-1}$. Since $y\geq 2$ because $\left\vert H\right\vert
\geq 4$, this is clearly satisfied if $x\geq y+2$, which is equivalent to $%
\left\vert G:H\right\vert \geq 4$.
\end{proof}

\begin{lemma}
\label{Lem_|Talpha|>=4}Let $G$ be an almost simple group with socle $T$. Let 
$M$ be a maximal subgroup of $G$ and let $M_{0}:=T\cap M$. Then $\left\vert
M_{0}\right\vert \geq 6$.
\end{lemma}

\begin{proof}
We can assume that $T\nleq M$. Since $G$ is almost simple, we have $%
M_{0}\neq 1$ (\cite[Theorem 1.3.6]{BrayHoltRDougal}) whence $\left\vert
M_{0}\right\vert \geq 2$. Moreover, $M_{0}\trianglelefteq M$, so by
maximality of $M$, the fact that $T$ is simple, and $1<M_{0}<T$, we get that 
$M=N_{G}\left( M_{0}\right) $ and $M_{0}=M\cap T=N_{T}\left( M_{0}\right) $.
Suppose, by contradiction, that $2\leq \left\vert M_{0}\right\vert \leq 5$.
Then $M_{0}$ is contained in a Sylow $p$-subgroup $P$ of $T$ where $p\in
\left\{ 2,3,5\right\} $ according to the case. If $M_{0}<P$ then $%
M_{0}<N_{P}\left( M_{0}\right) \leq N_{T}\left( M_{0}\right) =M_{0}$ - a
contradiction. Thus $M_{0}$ is a Sylow $p$-subgroup of $T$. But $2\leq
\left\vert M_{0}\right\vert \leq 5$ implies that $M_{0}$ is abelian so $%
M_{0}\leq C_{T}\left( M_{0}\right) \leq N_{T}\left( M_{0}\right) =M_{0}$.
Thus $M_{0}=Z\left( N_{T}\left( M_{0}\right) \right) $, and by Burnside's $p$%
-complement theorem (\cite[10.21]{Rose1978groups}), $M_{0}$ has a normal $p$%
-complement in $T$ - a contradiction since $T$ is simple.
\end{proof}

The following lemma is an easy corollary to a major result of \cite%
{GM2012products}. Let $x^{G}$ denote the conjugacy class of $x$ in $G$.

\begin{lemma}
\label{3cl} Let $T$ be a non-abelian simple group. Then there exist $\alpha
,\beta \in T$ such that $T=\alpha ^{T}\beta ^{T}S$, where $S$ is any subset
of $T$ of size at least $2$. In particular, there exist $\alpha ,\beta \in T$
such that $T=\alpha ^{T}\beta ^{T}\gamma ^{T}$ where $\gamma :=\beta
^{-1}\alpha ^{-1}$.
\end{lemma}

\begin{proof}
By \cite[Theorem 1.4]{GM2012products} there exist $\alpha ,\beta \in T$ with 
$\alpha ^{T}\beta ^{T}\cup \{1\}=T$. If $\alpha ^{T}\beta ^{T}=T$ then we
are done. Otherwise, $\alpha ^{T}\beta ^{T}=T-\left\{ 1\right\} $, and since
for any $s\in T$ we have $\left( T-\left\{ 1\right\} \right) s=T-\left\{
s\right\} $, we get that for any $s_{1}\neq s_{2}\in S$ we have $\alpha
^{T}\beta ^{T}s_{1}\cup \alpha ^{T}\beta ^{T}s_{2}=T$ and $T=\alpha
^{T}\beta ^{T}S$ follows. For proving $T=\alpha ^{T}\beta ^{T}\gamma ^{T}$
(for the same choice of $\alpha ,\beta \in T$) we can assume $\alpha
^{T}\beta ^{T}=T-\left\{ 1\right\} $. Hence $\gamma \neq 1$, implying $%
\left\vert \gamma ^{T}\right\vert \geq 2$. Now $T=\alpha ^{T}\beta
^{T}\gamma ^{T}$ follows by taking $S=\gamma ^{T}$ in the first claim.
\end{proof}

\begin{notation}
We denote by $\gamma _{\func{cp}}^{H}(G)$ the minimal positive integer $m$
such that there exist $m$ conjugates of $H\leq G$ whose product is $G$ ($%
\gamma _{\func{cp}}^{H}(G)=\infty $ if $G$ is not a product of conjugates of 
$H$).
\end{notation}

For the proof of Theorem \ref{Th_GeneralPrimitive} we use the classification
of finite primitive permutation groups as given by the O'Nan-Scott theorem,
for which we adopt the formulation and notation of \cite{LPS1988ON}. Thus $G$
is assumed to be a primitive permutation group on a set $\Omega $ of size $%
n=|G:H|$ where $H=G_{\alpha }$ is the stabilizer of some $\alpha \in \Omega $%
. The socle of $G$ is denoted $B\cong T^{k}$ with $k\geq 1$, where $T$ is a
simple group. Since $B$ acts transitively on $\Omega $ (being a non-trivial
normal subgroup of a primitive group), we have $G=BG_{\alpha }=BH$. Suppose
that $B$ is contained in the product of $t$ conjugates of $H$. Then $G$ is a
product of $t$ conjugates of $H$ (see \cite[Lemma 7(2)]{GaronziLevy}).
Moreover, since $B_{\alpha }=B\cap H\leq H$, we get that $B$ is certainly
contained in the product of $\gamma _{\func{cp}}^{B_{\alpha }}(B)$
conjugates of $H$. These considerations show that $\gamma _{\func{cp}%
}^{G_{\alpha }}(G)\leq \gamma _{\func{cp}}^{B_{\alpha }}(B)$ while $%
n=\left\vert G:G_{\alpha }\right\vert =\left\vert B:B_{\alpha }\right\vert $
and so in the cases where $B$ does not act regularly on $\Omega $ we will
prove our claim by exhibiting a suitable upper bound on $\gamma _{\func{cp}%
}^{B_{\alpha }}(B)$ (if $B$ acts regularly, $B_{\alpha }=1$ and $\gamma _{%
\func{cp}}^{B_{\alpha }}(B)=\infty $). Also note that since for any integer $%
m$ there are only finitely many isomorphism types of finite groups $A$ such
that $\left\vert A\right\vert \leq m$, for all primitive groups $G$
satisfying $\left\vert G\right\vert \leq m$ we get that $\gamma _{\func{cp}%
}^{B_{\alpha }}(B)$ is bounded above by some constant depending on $m$, and
so we may assume, that $\left\vert B\right\vert =\left\vert T\right\vert
^{k}>m$ for any fixed choice of $m$.

\section{\textbf{Type }\textrm{I}. $G$ is an affine primitive permutation
group}

\begin{proposition}
\label{Prop_Affine}Let $G$ be an affine primitive permutation group with a
non-trivial point stabilizer $H$. Then $G$ is a product of at most $%
1+c_{A}\log \left\vert G:H\right\vert $ conjugates of $H$, where $%
0<c_{A}\leq 3/\log 5<1.3$ is a universal constant.
\end{proposition}

In order to prove Proposition \ref{Prop_Affine}, we review some basic
properties of affine primitive permutation groups. If $G$ is an affine
primitive permutation group, then it has exactly one minimal normal subgroup 
$V$, which is abelian so $V\cong {C_{p}^{l}}$ for some prime $p$ and some
natural number $l$. Moreover $G=VH$ and, viewing $V$ as the vector space
over ${\mathbb{F}_{p}}$, then $H$ acts by conjugation irreducibly as a group
of linear transformations on $V$. When convenient we will use additive
notation for $V$.

\begin{lemma}
\label{Lem_trick} Let $G$ be an affine primitive permutation group with
point stabilizer $H$ and minimal normal subgroup $V\cong {C_{p}^{l}}$. Let $%
h\in H$ and $v\in V$. Set $w:=v^{h^{-1}}-v$ and $k:=\lceil \log p\rceil $.
Then $\left\langle w\right\rangle $ is contained in a product of $k+1$
conjugates of $H$.
\end{lemma}

\begin{proof}
We can assume $w\neq 0_{V}=1_{G}$ for which the claim clearly holds. Then $w$
is of order $p$, and any element of $\left\langle w\right\rangle $ is of the
additive form $sw$ where the integer $s$ satisfies $0\leq s\leq p-1$. Since $%
k:=\lceil \log p\rceil $, the base 2 representation of $s$ takes the form $%
s=\tsum\limits_{j=0}^{k-1}b_{j}2^{j}$ ($b_{j}\in \left\{ 0,1\right\} $ for
all $0\leq j\leq k-1$). Now note that $w=v^{h^{-1}}-v=v^{-1}hvh^{-1}\in
H^{v}H$. Similarly, for any $c\in {\mathbb{F}_{p}}$ we have $cw=\left(
cv\right) ^{h^{-1}}-cv\in H^{cv}H$. Thus, identifying the powers $2^{j}$
with elements of ${\mathbb{F}_{p}}$, we see that $sw\in $ $\left(
H^{v}H\right) \left( H^{2v}H\right) \left( H^{2^{2}v}H\right) \cdots \left(
H^{2^{k-1}v}H\right) $, for any $0\leq s\leq p-1$, where we pick $0_{V}$
from the $j$-th factor $\left( H^{2^{j}v}H\right) $ in the product if $%
b_{j}=0$ and $2^{j}w$ if $b_{j}=1$. However, also note that since $V$ is
abelian, $\left( H^{2^{j}v}H\right) \cap V$ is invariant under conjugation
by any element of $V$. Hence, for any choice of $u_{0},...,u_{k-2}\in V$ we
have 
\begin{equation*}
sw\in \Pi _{H}:=\left( H^{v}H\right) ^{u_{0}}\left( H^{2v}H\right)
^{u_{1}}\left( H^{2^{2}v}H\right) ^{u_{2}}\cdots \left( H^{2^{k-2}v}H\right)
^{u_{k-2}}\left( H^{2^{k-1}v}H\right) \text{.}
\end{equation*}%
Finally, for the choice $u_{k-2}=2^{k-1}v$, $u_{k-3}=u_{k-2}+2^{k-2}v$ and
in general $u_{k-j}=u_{k-j+1}+2^{k-j+1}v$ for all $2\leq j\leq k$ where $%
u_{k-1}:=0_{V}$,\ we get that $\Pi _{H}$ is equal to a product of $k+1$
conjugates of $H$.
\end{proof}

\begin{lemma}
\label{Lem_maxf(p)}For each prime number $p$ define $f(p):=\lceil \log
p\rceil /\log p$. Then $f(p)$ has a global maximum at $p=5$. Consequently 
\begin{equation}
\lceil \log p\rceil \leq (3/\log 5)\log p\text{, for every prime }p\text{.}
\label{Ineq_log2p}
\end{equation}
\end{lemma}

\begin{proof}
First check that $1+1/\log 11<1.29<3/\log 5$. Then, using this, we get: 
\begin{equation*}
f(p)\leq (\log p+1)/\log p=1+1/\log p<3/\log 5=f(5)\text{, }\forall p\geq 11%
\text{,}
\end{equation*}%
and for $p=2,3,7$ we verify explicitly that $f\left( p\right) <f\left(
5\right) $. Hence $f(p)$ has a global maximum $f\left( 5\right) =3/\log 5$
at $p=5$. Finally, $\lceil \log p\rceil =f(p)\log p$ $\leq f(5)\log p$.
\end{proof}

\begin{proof}[\textbf{Proof of Proposition \protect\ref{Prop_Affine}}]
Using the notation introduced after the statement of the proposition, $\log
|G:H|~=$ $\log |V|~=$ $\log p^{l}=l\log p$. Using Inequality \ref{Ineq_log2p}%
, we obtain:%
\begin{equation*}
1+l\lceil \log p\rceil \leq 1+(3/\log 5)l\log p=1+(3/\log 5)\log |G:H|\text{.%
}
\end{equation*}

Thus, it is enough to show that $G$ is a product of at most $1+$ $l\lceil
\log p\rceil $ conjugates of $H$.

Fix a non-zero vector $v\in V$. If $v$ is central in $G$ then $V=\langle
v\rangle $ by minimality of $V$. It follows that $H$ is a non-trivial normal
subgroup of $HV=G$ since $V$ is central - a contradiction to $H$ being
core-free. Therefore $v$ is not central, and there is some $h\in H$ with $%
v^{h^{-1}}\neq v$. Set $w:=v^{h^{-1}}-v$.

We claim that there are $l$ elements $h_{1},\ldots ,h_{l}\in H$ such that $%
B:=\{w^{h_{1}},\ldots ,w^{h_{l}}\}$ is a vector space basis of $V=C_{p}^{l}$%
. Note that since $w\neq 0_{V}$, this claim is immediate for $l=1$, and
hence we assume $l\geq 2$. Suppose by contradiction that $1\leq m<l$ is the
maximal integer such that there exist $h_{1},h_{2},\ldots ,h_{m}\in H$ for
which $B=\{w^{h_{1}},\ldots ,w^{h_{m}}\}$ is linearly independent. It
follows that for any $h\in H$, $w^{h}\in Span\left( B\right) $. Thus $%
Span\left( B\right) =Span\left( \left\{ w^{h}|h\in H\right\} \right) $. This
shows that $Span\left( B\right) $ is a proper non-trivial $H$-invariant
subspace of $V$, contradicting the fact that $H$ acts irreducibly on $V$.
Thus there exists a basis of $V$ of the form $B:=\{w^{h_{1}},\ldots
,w^{h_{l}}\}$.

For each $v\in V$ there exist $s_{1},...,s_{l}\in {\mathbb{F}_{p}}$ for
which $v=\Sigma _{i=1}^{l}s_{i}w^{h_{i}}$. Applying Lemma \ref{Lem_trick} to
each $w^{h_{i}}$ separately, we get that each $v\in V$ belongs to $\Pi
_{1}\cdots \Pi _{l}$, where each $\Pi _{i}$ is a product of $\lceil \log
p\rceil +1$ conjugates of $H$. But, as in the proof of Lemma \ref{Lem_trick}%
, this shows that $V\subseteq \Pi _{1}^{u_{1}}\cdots \Pi _{l-1}^{u_{l-1}}\Pi
_{l}$ for any choice of $u_{1},...,u_{l-1}\in V$, and one can choose these
elements so that the product $\Pi _{1}^{u_{1}}\cdots \Pi _{l-1}^{u_{l-1}}\Pi
_{l}$ is a product of at most $l\lceil \log p\rceil +1$ conjugates of $H$.
\end{proof}

From Proposition \ref{Prop_Affine} it follows that if $G$ is an affine
primitive permutation group then $\gamma _{\func{cp}}^{H}(G)\leq c_{I}\log n$
where the constant $c_{I}$ satisfies $0<c_{I}<2.3$.

\section{\textbf{Type }\textrm{II}. $G$ is an almost simple primitive
permutation group\label{Section_almost}}

In this case we have $k=1$ and $B=T$. Note that $T$ is a non-abelian simple
group acting transitively on $\Omega $. Furthermore, $T$ does not act
regularly on $\Omega $ by \cite{LPS1988ON}.

First suppose that $|G|<n^{9}$. By \cite[Theorem 3]{LNS2}, since $T_{\alpha
} $ is a subset of $T$ of size at least $2$ (because $T$ does not act
regularly), there exists a constant $c_{1}$ such that $T$ is a product of
less than $c_{1}\log |T|$ conjugates of $T_{\alpha }$. Now $|T|\leq
|G|<n^{9} $ implies that $\gamma _{\mathtt{cp}}^{T_{\alpha }}(T)<$ $9\cdot
c_{1}\log n$.

Assume that $|G|\geq n^{9}$. By \cite{LiebeckMinDeg1984} one of the
following holds:

\begin{enumerate}
\item $T=A_{m}$, where $m\geq 5$ and either

\begin{enumerate}
\item $\Omega $ is the set of all subsets of size $k$ of $\{1,\ldots ,m\}$, $%
n=\binom{m}{k}$ or

\item $\Omega $ is the set of all partitions of $\{1,\ldots ,m\}$ into $a$
subsets of size $b$ where $ab=m$, $a>1$, $b>1$; $n=m!/({(b!)}^{a}a!)$.
\end{enumerate}

\item $T$ is a classical simple group acting on an orbit of subspaces of the
natural module, or (in the case $T=PSL\left( d,q\right) $) on pairs of
subspaces of complementary dimensions.
\end{enumerate}

Since $n=\left\vert G:G_{\alpha }\right\vert =\left\vert T:T_{\alpha
}\right\vert $, and since $G$ is almost simple, we have by \cite[Lemma 2.7
(i)]{AschGur1989} that\ $|G:T|\leq |$\textrm{Out}$(T)|<n$. This gives $%
\left\vert T\right\vert >n^{8}=\left( \frac{\left\vert T\right\vert }{%
\left\vert T_{\alpha }\right\vert }\right) ^{8}$, implying $\frac{\log
\left\vert T\right\vert }{\log \left\vert T_{\alpha }\right\vert }<\frac{8}{7%
}<2$. If $T_{\alpha }$ is maximal in $T$, we can conclude from \cite[Theorem
2]{LNS} that there exist a universal constant $c_{2}$ and a universal
function $f:\mathbb{N}\rightarrow \mathbb{N}$, such that for all $T$
satisfying $\left\vert T\right\vert >f\left( 2\right) $ it holds that $%
\gamma _{\mathtt{cp}}^{T_{\alpha }}(T)\leq $ $c_{2}\frac{\log \left\vert
T\right\vert }{\log \left\vert T_{\alpha }\right\vert }$. Now we claim that
this conclusion is in fact valid even if $T_{\alpha }$ is not maximal in $T$%
. More precisely, we claim that \cite[Theorem 2]{LNS} is valid for all
subgroups belonging to the families listed in \cite[Lemma 3.1]{LNS} in the
case $T=A_{m}$, and in \cite[Lemma 4.3]{LNS} in the case that $T$ is a
classical group. Note that these families include the $\left( T,T_{\alpha
}\right) $ of \cite{LiebeckMinDeg1984} listed above. Our claim is based on a
close examination of the use of the maximality assumption in the proof of 
\cite[Theorem 2]{LNS}. We find that the maximality assumption is used only
in two places. First, in appealing to \cite[Theorem 1]{LNS} in order to
discard cases of simple groups of Lie type of small Lie rank. Here we
replace \cite[Theorem 1]{LNS} by \cite[Theorem 1.3]{GPSSBoundedProof2013},
which applies to any subset of $T$ of size at least $2$. The second use of
the maximality assumption is to identify the possible isomorphism types for
maximal subgroups of the remaining simple groups, according to the
O'Nan-Scott classification in the alternating case and the Aschbacher
classification in the classical case. These are precisely the families
listed in \cite[Lemma 3.1]{LNS} and in \cite[Lemma 4.3]{LNS}. The rest of
the proof of Theorem 2 of \ \cite{LNS} carries through even when the
subgroups in question are not actually maximal.

Finally, by Lemma \ref{Lem_LNSTranslateToIndex} and Lemma \ref%
{Lem_|Talpha|>=4} we get that $\gamma _{\mathtt{cp}}^{H}(G)\leq c_{II}\log n$
for some universal constant $c_{II}>0$, for all primitive almost simple $G$.

\section{\textbf{Type \textrm{III}(a)}. $G$ is a primitive permutation group
of diagonal type.}

Here $B_{\alpha }$ is the diagonal subgroup of $B$ ($\Delta $ in the
notation of Proposition \ref{Prop_diagonal} below) and $n=|G:G_{\alpha
}|=\left\vert T\right\vert ^{k-1}$, where $k\geq 2$.

\begin{proposition}
\label{Prop_diagonal} Let $T$ be a non-abelian simple group, $k$ a positive
integer, $B:=T^{k}$. Set $\Delta :=\{(t,t,\ldots ,t)\ :\ t\in T\}\leq B$.
Then $k\leq \gamma _{\mathtt{cp}}^{\Delta }(B)\leq 3k-2$.
\end{proposition}

\begin{proof}
Suppose $B$ is a product of $m$ conjugates of $\Delta $. Then $\Delta \cong
T $ implies $|T|^{k}=|B|~\leq |\Delta |^{m}=|T|^{m}$. This proves $k\leq
\gamma _{\mathtt{cp}}^{\Delta }(B)$. For proving $\gamma _{\mathtt{cp}%
}^{\Delta }(B)\leq 3k-2$, choose $\alpha ,\beta ,\gamma \in T$ as in Lemma %
\ref{3cl}. Set $a:=\alpha ^{-1}$ and $b:=\gamma $. Then $T=\alpha ^{T}\beta
^{T}\gamma ^{T}=(a^{-1})^{T}(ab^{-1})^{T}b^{T}$.

Let $i\in \{1,\ldots ,k\}$. Let $\tau _{i}:T\rightarrow T^{k}$ be the map
that sends $t\in T$ to the element of $T^{k}$ that has $t$ in the $i$-th
component and $1$ elsewhere. We denote $T_{i}:=\tau _{i}(T)$. Consider $%
D_{i}:=\Delta \Delta ^{\tau _{i}(a)}\Delta ^{\tau _{i}(b)}\Delta $. We prove
that $D_{i}\supseteq T_{i}$. An element of $D_{i}$ has the form 
\begin{equation*}
(xyzw,\ xyzw,\ \ldots ,\ xyzw,\ \underset{i\text{-th entry}}{\underbrace{%
xy^{a}z^{b}w}},\ xyzw,\ \ldots ,\ xyzw)
\end{equation*}%
where $x,y,z,w\in T$ are arbitrary. In order to prove that $D_{i}\supseteq
T_{i}$, choose arbitrary $x,y,z\in T$ and $w=(xyz)^{-1}$. Then for the $i$%
-th component we have 
\begin{gather*}
xy^{a}z^{b}w=xy^{a}z^{b}(xyz)^{-1}=xa^{-1}yab^{-1}zbz^{-1}y^{-1}x^{-1} \\
=(xa^{-1}x^{-1})(\left( xy\right) \left( ab^{-1}\right) \left(
y^{-1}x^{-1}\right) )(\left( xyz\right) b\left( z^{-1}y^{-1}x^{-1}\right) )
\\
\in (a^{-1})^{T}(ab^{-1})^{T}b^{T}=T\text{.}
\end{gather*}%
Since $\left\{ \left( x,xy,xyz\right) |x,y,z\in T\right\} =T^{3}$, we can
deduce $D_{i}\supseteq T_{i}$.

It follows that $B=T_{1}\cdots T_{k}=\Delta T_{2}\cdots T_{k}\subseteq
\Delta D_{2}\cdots D_{k}=D_{2}\cdots D_{k}.$ Therefore 
\begin{eqnarray}
B &=&D_{2}\cdots D_{k}=(\Delta \Delta ^{\tau _{2}(a)}\Delta ^{\tau
_{2}(b)}\Delta )\cdots (\Delta \Delta ^{\tau _{k}(a)}\Delta ^{\tau
_{k}(b)}\Delta )  \notag \\
&=&(\Delta \Delta ^{\tau _{2}(a)}\Delta ^{\tau _{2}(b)})\cdots (\Delta
\Delta ^{\tau _{k}(a)}\Delta ^{\tau _{k}(b)})\Delta \text{,}  \notag
\end{eqnarray}%
and $B$ is a product of $3(k-1)+1=3k-2$ conjugates of $\Delta $.
\end{proof}

By Proposition \ref{Prop_diagonal} we have $\gamma _{\mathtt{cp}}^{\Delta
}(B)\leq 3k-2$. On the other hand 
\begin{equation*}
\log |G:G_{\alpha }|~=\left( k-1\right) \log |T|~\geq \left( k-1\right) \log
60>5\left( k-1\right) \text{.}
\end{equation*}%
Comparing the numbers we see that $B$ is the product of less than $\log
|G:G_{\alpha }|$ conjugates of $B_{\alpha }$, and so we have $\gamma _{%
\mathtt{cp}}^{H}(G)\leq c_{III(a)}\log n$ with $0<c_{III(a)}\leq 1$.

\section{\textbf{Type \textrm{III}(b)}. $G$ is a primitive permutation group
of product action type.}

Let $R$ be a primitive permutation group of type II or III(a) on a set $%
\Gamma $. For $\ell >1$, let $W=R\wr S_{\ell }$, and take $W$ to act on $%
\Omega =\Gamma ^{\ell }$ in its natural product action. Then for $\gamma \in
\Gamma $ and $\alpha =(\gamma ,\ldots ,\gamma )\in \Omega $ we have $%
W_{\alpha }=R_{\gamma }\wr S_{\ell }$, and $n={|\Gamma |}^{\ell }$. If $K$
is the socle of $R$ then the socle $B$ of $W$ is $K^{\ell }$, and $B_{\alpha
}={(K_{\gamma })}^{\ell }\not=1$. If $G$ is primitive of type III(b), then $%
G $ satisfies $B\leq G\leq W$ and acts transitively on the $\ell $ factors
of $B=K^{\ell }$. In particular, $soc(G)=soc(W)=K^{\ell }$. By the
discussion of cases II and III(a) we know that $K$ is the product of at most 
$\max \left\{ c_{II},c_{III(a)}\right\} \cdot \log |K:K_{\gamma }|$
conjugates of $K_{\gamma }$. Since $B=K^{\ell }$ and $B_{\alpha }={%
(K_{\gamma })}^{\ell }$, we get that $B$ is the product of at most $\max
\left\{ c_{II},c_{III(a)}\right\} \cdot \log |K:K_{\gamma }|$ conjugates of $%
B_{\alpha }$. Now $\left\vert G:G_{\alpha }\right\vert =\left\vert \Gamma
\right\vert ^{\ell }$, and, since $K$ acts transitively on $\Gamma $, $%
\left\vert \Gamma \right\vert =|K:K_{\gamma }|$. Hence 
\begin{equation*}
\log n=\log \left\vert G:G_{\alpha }\right\vert =\log \left\vert \Gamma
\right\vert ^{\ell }=\ell \log |K:K_{\gamma }|\text{,}
\end{equation*}%
and we have proved that $\gamma _{\mathtt{cp}}^{H}(G)\leq c_{III(b)}\log n$
with $0<c_{III(b)}\leq \max \left\{ c_{II},c_{III(a)}\right\} $.

\section{\textbf{Type \textrm{III}(c)}. $G$ is a primitive permutation group
of twisted wreath product type.}

Let $P$ be a transitive permutation group of degree $k$, acting on $\left\{
1,...,k\right\} $, and let $Q\leq P$ be the stabilizer of $1$. Let $\varphi
:Q\rightarrow $\textrm{Aut}$(T)$ be a homomorphism such that $\varphi (Q)$
contains all the inner automorphisms of $T$. Let 
\begin{equation*}
B_{0}=\{f:P\rightarrow T\ :\ f(pq)=f(p)^{\varphi (q)}\ \forall p\in P,\ q\in
Q\}\text{.}
\end{equation*}%
Then $B_{0}$ is a group with pointwise multiplication. let $L=\left\{
l_{1},...,l_{k}\right\} \subseteq P$ be an arbitrary fixed left transversal
of $Q$ in $P$. By definition of $B_{0}$, a function $f\in B_{0}$ is
determined by its values on $L$. On the other hand, the values of $f$ on $L$
can be arbitrary, and therefore we get $B_{0}\cong T^{k}$. More
specifically, for $\ell \in L$ and $t\in T$ define $f_{t,\ell }:P\rightarrow
T\ $ by: 
\begin{equation*}
f_{t,\ell }(x):=\left\{ 
\begin{array}{cc}
t^{\varphi (\ell ^{-1}x)} & \text{if}\ x\in \ell Q, \\ 
1 & \text{if}\ x\not\in \ell Q%
\end{array}%
\right. \forall x\in P\text{.}
\end{equation*}%
We claim that $f_{t,\ell }\in B_{0}$. Indeed, let $p\in P,\ q\in Q$, and
consider $f_{t,\ell }\left( pq\right) $. If$\ pq\not\in \ell Q$ then $%
p\not\in \ell Q$ and $\ f(pq)=f(p)=1$ so $f(pq)=f(p)^{\varphi (q)}$ holds. If%
$\ pq\in \ell Q$ then $p\in \ell Q$ and there exists $q_{0}\in Q$ such that $%
p=\ell q_{0}$. Hence 
\begin{align*}
f_{t,\ell }\left( pq\right) & =f_{t,\ell }\left( \ell q_{0}q\right)
=t^{\varphi (\ell ^{-1}\ell q_{0}q)}=t^{\varphi (q_{0}q)}=t^{\varphi
(q_{0})\varphi (q)} \\
& =\left( t^{\varphi (q_{0})}\right) ^{\varphi (q)}=f_{t,\ell }\left( \ell
q_{0}\right) ^{\varphi (q)}=f(p)^{\varphi (q)}\text{.}
\end{align*}%
Furthermore, if $\ell =l_{i}$ then $f_{t,\ell }$ corresponds to the element
of $T^{k}$ that has $t$ in the $i$-th component and $1$ elsewhere. To see
this we just have to check that $f_{t,l_{i}}\left( l_{j}\right) $ satisfies $%
f_{t,l_{i}}\left( l_{j}\right) =t$ if $i=j$ and $f_{t,l_{i}}\left(
l_{j}\right) =1$ if $i\neq j$ and this is immediate from the definition.
Thus we can construct an explicit isomorphism $B_{0}\rightarrow T^{k}$ which
maps $\left\{ f_{t,l_{i}}|t\in T\right\} $ onto $T_{i}$, where $T_{i}$ is
the $i$-th direct factor of $T^{k}$, $1\leq i\leq k$. From now on we
identify $\left\{ f_{t,l_{i}}|t\in T\right\} $ with $T_{i}$. Furthermore, $P$
acts on $B_{0}$ in the following way: if $f\in B_{0}$ and $p\in P$ define $%
f^{p}(x):=f(px)$ for all $x\in P$. The semidirect product $G:=B_{0}\rtimes P$
with respect to this action is called the twisted wreath product (of $P$ and 
$T$). Then $G$ acts transitively by right multiplication on the set $\Omega $
of size $n=|B_{0}|$ of all right cosets of $P$. This action is not always
primitive. If it is, $G$ belongs to class III(c). In this case $B=B_{0}\cong
T^{k}=T_{1}\times \cdots \times T_{k}$ is a normal subgroup (the unique
minimal one) in $G$ and acts regularly on $\Omega $, and we take $G_{\alpha
}=P$. We have $n=\left\vert T\right\vert ^{k}$.

Set, for each $1\leq i\leq k$, $Q_{i}:=l_{i}Ql_{i}^{-1}$. We prove that $%
Q_{i}$ leaves $T_{i}$ invariant with respect to the action of $P$ on $B_{0}$%
. For this we have to show that if $p\in Q_{i}$, namely, $p=l_{i}ql_{i}^{-1}$
for some $q\in Q$, then, for all $x\in P$, $x\in l_{i}Q$ if and only if $%
px\in l_{i}Q$. But $px \in l_iQ$ means $l_iql_i^{-1} x \in l_iQ$, and this
is true if and only if $ql_i^{-1}x \in Q$, which, since $q \in Q$, is
equivalent to $l_i^{-1}x \in Q$, which is equivalent to $x \in l_iQ$.

Thus the action of $P$ on $B_{0}$ induces an action of $Q_{i}$ on $T_{i}$
for each $1\leq i\leq k$. Note that for $p=l_{i}ql_{i}^{-1}\in Q_{i}$ and $%
x=l_{i}q_{0}\in l_{i}Q$, we get 
\begin{align*}
f_{t,l_{i}}^{p}\left( x\right) =f_{t,l_{i}}\left( px\right) =t^{\varphi
(l_{i}^{-1}px)}=t^{\varphi (l_{i}^{-1}l_{i}ql_{i}^{-1}l_{i}q_{0})} \\
= t^{\varphi (qq_{0})}=\left( t^{\varphi (q)}\right) ^{\varphi
(q_{0})}=f_{t^{\varphi (q)},l_{i}}\left( x\right) \text{.}
\end{align*}
Since $f_{t,l_{i}}^{p}\left( x\right) =f_{t^{\varphi (q)},l_{i}}\left(
x\right) $ clearly holds also for any $x\notin l_{i}Q$ we get $%
f_{t,l_{i}}^{p}=f_{t^{\varphi (q)},l_{i}}$ and so $Q_{i}$ acts on $T_{i}$ as 
$\varphi \left( Q\right) $.

By Lemma \ref{3cl} there exist $t_{i,1},t_{i,2},t_{i,3}\in T_{i}$ for each $%
1\leq i\leq k$ such that $%
T_{i}=t_{i,1}^{T_{i}}t_{i,2}^{T_{i}}t_{i,3}^{T_{i}} $. For each $j\in
\{1,2,3\}$ set $O_{i,j}:=\left\{ p^{-1}t_{i,j}p|p\in Q_{i}\right\} $. In $%
G:=B_{0}\rtimes P$, the action of $P$ on $B_{0}$ is the restriction of the
conjugation action of $G$ on itself, and hence $O_{i,j}$ is the orbit of $%
t_{i,j}$ under the action of $Q_{i}$ on $T_{i}$. Since, by assumption, 
\textrm{Inn}$\left( T\right) \leq \varphi \left( Q\right) $, $O_{i,j}$ is a
normal subset of $T_{i}$, and in particular contains $t_{i,j}^{T_{i}}$, the
conjugacy class of $t_{i,j}$ in $T_{i}$. Set $X_{i,j}:=t_{i,j}^{-1}O_{i,j}$, 
$1\leq j\leq 3$. Since the $O_{i,j}$'s are normal sets,%
\begin{equation*}
T_{i}=t_{i,1}^{-1}t_{i,2}^{-1}t_{i,3}^{-1}t_{i,1}^{T_{i}}t_{i,2}^{T_{i}}t_{i,3}^{T_{i}}\subseteq t_{i,1}^{-1}t_{i,2}^{-1}t_{i,3}^{-1}O_{i,1}O_{i,2}O_{i,3}=X_{i,1}X_{i,2}X_{i,3}%
\text{.}
\end{equation*}%
Moreover $X_{i,j}=\{t_{i,j}^{-1}p^{-1}t_{i,j}p\ :\ p\in Q_{i}\}\subseteq
P^{t_{i,j}}P$ for $j=1,2$, and $%
X_{i,3}=t_{i,3}^{-1}O_{i,3}=O_{i,3}t_{i,3}^{-1}=\{p^{-1}t_{i,3}pt_{i,3}^{-1}%
\ :\ p\in Q_{i}\}\subseteq PP^{t_{i,3}^{-1}}$. It follows that 
\begin{equation*}
T_{i}=X_{i,1}X_{i,2}X_{i,3}\subseteq
P^{t_{i,1}}PP^{t_{i,2}}PP^{t_{i,3}^{-1}}.
\end{equation*}%
Thus $B=T_{1}\cdots T_{k}$ is contained in a product of at most $5k$
conjugates of $P$, and since $n=|T|^{k}$ we have $k=\log n/\log \left\vert
T\right\vert $. Hence 
\begin{equation*}
\gamma _{\mathtt{cp}}^{P}(G)\leq 5k=5\log n/\log |T|\leq \frac{5}{\log 60}%
\log n<\log n\text{.}
\end{equation*}%
We have proved that if $G$ is a primitive group of twisted wreath product
type then $\gamma _{\mathtt{cp}}^{H}(G)\leq c_{III(c)}\log n$ where $%
0<c_{III(c)}<1$.

This completes the proof of Theorem \ref{Th_GeneralPrimitive}.

{}\smallskip \textbf{Acknowledgement}: We would like to thank N. Nikolov for
a useful discussion.

\end{document}